\documentclass[12 pt]{amsart}
\usepackage[all]{xy}

\usepackage[urlcolor=blue]{hyperref} 

\usepackage{enumitem} 
\usepackage{xspace} 

\usepackage{amsfonts}
\usepackage{amsmath,amssymb}      
\usepackage{amsthm}
\usepackage{pdfsync}       
\usepackage{color}
\usepackage[english]{babel}

\usepackage[T1]{fontenc}

\theoremstyle{plain}
\newtheorem{theorem}{Theorem}[section]
\newtheorem{proposition}[theorem]{Proposition}
\newtheorem{lemma}[theorem]{Lemma}
\newtheorem{corollary}[theorem]{Corollary}

\theoremstyle{definition}
\newtheorem{example}[theorem]{Example}
\newtheorem{definition}[theorem]{Definition}
\newtheorem{remark}[theorem]{Remark}

\newcommand \PP {{\mathbb P}}
\newcommand \QQ {{\mathbb Q}}

\newcommand \NN {{\mathbb N}}
\newcommand \M{{\mathcal M}}
\newcommand \gin   {{\rm gin}}
\newcommand \rgin  {{\rm rgin}}

\newcommand \KK {{K}}
\newcommand \sat  {{\rm sat}}
\newcommand \reg  {{\rm reg}}

\newcommand \HS  {{\rm HS}}
\newcommand \ideal[1] {\langle #1 \rangle}

\newcommand \ie  {\textit{i.e.}}

\def \X(#1){\{x_1,\dots, x_{#1}\}}

\def \GL    {{\rm GL}}

\def \gin   {{\rm gin}}
\def \LT{{\rm LT}}

\begin{document}

\title {Extremal Behaviour in Sectional Matrices}
\begin{abstract}
In this paper we recall the object \textit{sectional matrix}
which encodes the Hilbert functions of successive hyperplane sections of a homogeneous ideal.
We translate and/or reprove recent results in this language. 
Moreover, some new results are shown about their maximal growth,
in particular a new generalization of Gotzmann's Persistence Theorem,
the presence of a GCD for a truncation of the ideal,
and applications to saturated ideals.

\end{abstract}

\author{Anna Bigatti}
\address{Department of Mathematics, Genoa University, Via Dodecaneso 35, 16146 Genoa, Italy.}
\email{bigatti@dima.unige.it}
\author{Elisa Palezzato} 
\address{Department of Mathematics, Genoa University, Via Dodecaneso 35, 16146 Genoa, Italy.}
\email{palezzato@dima.unige.it}
\author{Michele Torielli}
\address{Department of Mathematics, Hokkaido University, Kita 10, Nishi 8, Kita-Ku, Sapporo 060-0810, Japan.}
\email{torielli@math.sci.hokudai.ac.jp}






%
%

\date{\today}
\maketitle


\section{Introduction}

Let $K$ be a field of characteristic 0 and let $P$
be the polynomial ring
$\KK[x_1,\dots,x_n]$ 
 with $n$ indeterminates and the standard grading.
Given a finitely generated $\mathbb{Z}$-graded $P$-module $M=\oplus_{d \in \NN}M_d$,
the $M_d$'s are finite-dimensional $K$-vector spaces.
The \textbf{Hilbert function} of $M$,
$
H_M\colon\mathbb{Z} \to {\NN} 
\text{ \ \ with \ \ } 
H_M(d):= \dim_K(M_{d})
$,
 is a very frequent and powerful tool of investigation in
Commutative Algebra.

  From Macaulay \cite{Macaulay27}, it is well known that the
  computation of the Hilbert function of $M$ may be reduced
  to the computation of the Hilbert function of some
  $K$-algebras of type $P/J$, where $J$ is a monomial ideal.  In
  particular, if $I$ is a homogeneous ideal in $P$, then 
  $H_{P/I} = H_{P/\LT(I)}$.


Another very common practice 
 consists of studying
generic hyperplane sections which, in algebraic terms, means 
reducing modulo by a generic linear form.
The combination of Hilbert functions and 
hyperplane sections lead to the result by Green
\cite{Green88} (1988).

The \textit{sectional matrix} of a homogenous ideal~$I$ was introduced
by Bigatti and Robbiano in \cite{bigatti1997borel} (1997) unifying the
concepts of the Hilbert function of 
a homogeneous ideal $I$ (along the rows)
 and of its hyperplane sections
(along the columns). 
  Sectional matrices did not receive much attention, and in this
  paper we want to revive them. 
  We extend some results in this language, confirming the merit of
  this tool and suggesting that further investigation might cast a
  new light on many aspects of Commutative Algebra.

\smallskip

In Section \ref{definitions} we set our notation and recall 
the definition of sectional matrix.
In Sections \ref{SecMat} we recall its main properties
converting the results from \cite{bigatti1997borel} 
into terms of the quotient $P/I$ (instead of the ideal $I$).
In particular, Theorem \ref{thm:HI_SM} shows  Macaulay's and Green's
inequalities.
In Section \ref{persistence}, 
we recall Gotzmann's Persistence Theorem and 
the sectional matrix analogue from \cite{bigatti1997borel}.
Moreover, we generalize it into a \textit{sectional} version (Theorem \ref{thm:Sectional_Persistence}).

In Section \ref{HP,HS}, we describe how to deduce information on the dimension and the degree of a homogeneous ideal
in terms of  the entries of its sectional matrix.
In Section \ref{GCD}, we 
show how the  extremal behaviour 
implies the
 presence of a GCD for the truncation of homogeneous ideals.
In Section \ref{saturated}, we 
apply these results to the class of saturated ideals. 
Finally,
in Section \ref{examples}, we present several examples 
comparing the information given by the sectional matrix, 
the generic initial ideal, and the resolution of a homogeneous ideals.

\smallskip
The examples in this paper have been computed with CoCoA
(\cite{CoCoA-5}, \cite{CoCoALib}, \texttt{SectionalMatrix,
  PrintSectionalMatrix}).



\section{Definitions and notation}
\label{definitions}

Let $K$ be a field of characteristic 0 and
 $P = \KK[x_1,\dots,x_n]$ be the polynomial ring with $n$
indeterminates with the standard grading. 
Let $I\subseteq P$ be a homogeneous ideal in $P$ and $A = P/I$.
Then $A$ is a graded $P$-module $\oplus_{d \in\NN} A_d$, 
where $A_d=P_d/I_d$.

The definition of the Hilbert function was extended in
\cite{bigatti1997borel} to the bivariate function encoding the Hilbert
functions of successive generic hyperplane sections: the \textbf{sectional
  matrix} of a homogeneous ideal $I$ in $P$.  In this paper,
we define, in the obvious way, the sectional matrix for
the quotient algebra $P/I$, 
and then we show how to adapt the results given 
in \cite{bigatti1997borel} to the use of $P/I$.

\begin{definition}\label{def:SM} 
Given a homogeneous ideal $I$ in $P=K[x_1,\dots,x_n]$, we define the
\textbf{sectional matrix} of $I$ and of $P/I$ to be the functions
$\{1,\dots,n\}\times \NN \longrightarrow \NN$
\begin{eqnarray*}
\M_I(i,d) &=& 
\dim_\KK((I+(L_1,\dots,L_{n-i}))/(L_1,\dots,L_{n-i}))_d,\\
\M_{P/I}(i,d) &=& \dim_\KK(P_d/(I+(L_1,\dots,L_{n-i}))_d),
\end{eqnarray*}
where $L_1,\dots,L_{n-i}$
 are generic linear forms.
Notice 
that $\M_{P/I}(n,d) = H_{P/I}(d)$
and 
$\M_{P/I}(i,d) =\binom{d+i-1}{i-1} - \M_I(i,d)$.
\end{definition}

\begin{remark}
A generic linear form is a polynomial
$
L = a_{1}x_1+\dots+a_{n}x_n
$
in $K(a_1,\dots,a_n)[x_1,...,x_n]$.
In this paper we restrict our attention to a field $K$ of
characteristic 0, so the equalities of the Definition~\ref{def:SM}
hold for any $L'=\alpha_{1}x_1+\dots+\alpha_{n}x_n $
with $(\alpha_1,\dots, \alpha_n)$ in a non-empty Zariski-open set in
$\PP_{K}^n$. Therefore in this case it is common practice to
talk about 
``generic linear forms in $K[x_1,\dots,x_n]$'' instead of
dealing with the explicit extension of $K$.
\end{remark}

This small example
will be used as a running example throughout the paper.
\begin{example}\label{ex:first}
Let $P=\mathbb{Q}[x,y,z]$ 
and $I = (x^4 -y^2z^2, \; xy^2 -yz^2 -z^3)$ an ideal of~$P$. 
Then the sectional matrix of $P/I$ is
$$
\begin{array}{rcccccccccc} 
          &  _0 & _1 & _2 & _3  & _4 
          & _5 & _6 & _7 & \dots\\
H_{P/(I+\ideal{L_1,L_2})}(d)= 
\M_{P/I}(1,d):&  1 & 1 & 1 & 0 & 0 & 0 & 0 & 0 & \dots\\
H_{P/(I+\ideal{L_1})}(d)= 
\M_{P/I}(2,d):&  1 & 2 & 3 & 3 & 2 & 1 & 0 & 0 & \dots\\
H_{P/I}(d)=\M_{P/I}(3,d):&  1 & 3 & 6 & 9 & 11 & 12 & 12 & 12 & \dots
\end{array} 
$$
where the continuations of the lines are obvious in this example.
The general theory about truncation and continuation of the lines
will be described in Theorem~\ref{thm:SM_Persistence} and Remark~\ref{rem:reg-trunc}.
\end{example}

\section{Background results on sectional matrices}
\label{SecMat}

In this section we recall the
main properties of sectional matrices from~\cite{bigatti1997borel}
translating them in terms of the quotient $P/I$.
In particular, we describe the persistence theorem and the connection with $\rgin$.

Let $\sigma$ be a term-ordering on  $P=\KK[x_1,\dots,x_n]$.
The \textbf{leading term ideal} or \textbf{initial ideal} of 
an ideal $I\subseteq P$ is the
  ideal, denoted $\LT_\sigma(I)$, generated by $\{\LT_\sigma(f) \mid f \in
  I{\setminus}\{0\}\;\}$.
For any homogenous ideal $I$
it is well known that $H_{P/I} = H_{P/\LT_\sigma(I)}$.  
This
nice property does not extend to $\M_{P/I}$, but only one inequality
holds, as Conca  proved in \cite{conca2003reduction}: 
we write his result in \textit{sectional matrix} notation.

\begin{theorem}[Conca, 2003]\label{thm:Conca}
  Let $I$ be a homogeneous ideal in $P=K[x_1,\dots,x_n]$ and
  $\sigma$ a term-ordering. Then,
for all $i=1,\dots, n$ and $d\in\NN$,
$\M_{P/I}(i,d)\le\M_{P/\LT_\sigma(I)}(i,d)$
\end{theorem}

\begin{example}\label{ex:conca}
Recall $I$ from Example~\ref{ex:first}.
For $\sigma=$DegRevLex
compare  $\M_{P/I}$ with  $\M_{P/\LT_\sigma(I)}$:
the $\sigma$-Gr\"obner basis of $I$ is
$\{xy^2 {-}yz^2 {-}z^3,  \\
\;  x^4 {-}y^2z^2,  
\;  x^3yz^2 {-}y^4z^2 {+}x^3z^3,
\;  y^5z^2 {-}y^4z^3 {+}x^3z^4 {-}x^2yz^4 {-}x^2z^5\}$,
thus we have $\LT_\sigma(I) = (xy^2,
\;  x^4,  
\; x^3yz^2,
\;  y^5z^2).$
$$
\begin{array}{rcccccccccc} 
          &  _0 & _1 & _2 & _3  & _4 
          & _5 & _6 & _7 & \dots\\
\M_{P/\LT_\sigma(I)}(1,d):&  1 & 1 & 1 & 0 & 0 & 0 & 0 & 0 & \dots\\
\M_{P/\LT_\sigma(I)}(2,d):&  1 & 2 & 3 & 3 & 2 & 1 & 1 & 0 & \dots\\
\M_{P/\LT_\sigma(I)}(3,d):&  1 & 3 & 6 & 9 & 11 & 12 & 12 & 12 & \dots
\end{array} 
$$
Then we observe that $\M_{P/I}(2,6)=0<1=\M_{P/\LT_\sigma(I)}(2,6)$.
Notice that the third lines are equal
for all term-orderings
because the Hilbert functions are the same:
$H_{P/I} = H_{P/\LT_\sigma(I)}$.
\end{example}

In this paper we compare some of our results with the ones from \cite{ahn2007some}. 
In order to make the comparison clearer to the reader we need to introduce the notion of $s$-reduction number
and to describe how it is stated in terms of the sectional matrix.
The definition of $s$-reduction number has several equivalent
formulations and we recall here the one given in  
\cite{ahn2007some}.

\begin{definition}
Let $I$ be a homogeneous ideal in $P= K[x_1,\dots,x_n]$.
The \textbf{$s$-reduction number}, $r_s(P/I)$, is 
$\max\{d \mid H_{P/(I + (L_1,\dots,L_s))}(d) {\ne} 0\}$, 
where $L_1,\dots,L_s$ are generic linear forms in $P$.
In our language
$$r_s(P/I) = \max\{d \mid \M_{P/I}(n{-}s,d) {\ne} 0\}.$$
The \textbf{reduction number} $r(P/I)$ is $r_{\dim(P/I)}(P/I)$.

Notice that, for their definitions, the reduction number and the sectional matrix,
use ``complementary'' indices $s$ and $n-s$.
\end{definition}

\begin{example}\label{ex:reduction-number}
In Example~\ref{ex:first} and Example~\ref{ex:conca}
we see that $I$ and $\LT_\sigma(I)$ 
have the same 2-reduction number,
$r_2(P/I)=r_2(P/\LT_\sigma(I))=2$,
and different 1-reduction number:
 $r_1(P/I)=5$ and $r_1(P/\LT_\sigma(I))=6$.
From the equalities $\dim(P/I) = \dim(P/\LT_\sigma(I))= 1$ it follows that
 $r(P/I){=}5$ and $r(P/\LT_\sigma(I)){=} 6$.
\end{example}

\begin{remark}
Using Theorem \ref{thm:Conca} Conca in \cite{conca2003reduction}
  proved the inequality for the reduction numbers:
 $r(P/I)\le r(P/\LT_\sigma(I))$.
\end{remark}


Going back to the problem of finding a monomial ideal with the same
sectional matrix as $P/I$, we recall the definitions of
\textit{strongly stable} ideal and of \textit{gin}.
A monomial ideal $J$ is said to be  \textbf{strongly stable} 
if for every power-product $t \in J$ and
every $i,j$ such that $i<j$ and $x_j|t$, 
the power-product $x_i{\cdot} t/x_j$ is in $J$.

In \cite{galligo1974propos} Galligo proved that,
given a homogeneous ideal $I$ in the polynomial ring $K[x_1,\dots,x_n]$, with $K$ a
field of characteristic $0$ and $\sigma$ a term-ordering such that
$x_1>_\sigma x_2 >_\sigma \dots >_\sigma x_n$,
 then there exists a non-empty Zariski-open set 
 $U\subseteq\GL(n)$ and a strongly stable ideal $J$ such that
for each  $g\in U$, $\LT_\sigma(g(I)) = J$. 
This ideal is called the \textbf{\boldmath generic initial ideal of $I$
 with respect to $\sigma$} and it is denoted by~{\boldmath  $\gin_\sigma(I)$}.  
 In particular, when $\sigma=$DegRevLex, it is denoted by~{\boldmath $\rgin(I)$}.


\begin{example}
Consider the ideal $I = (x^4 -y^2z^2, \; xy^2 -yz^2 -z^3)$ from Example~\ref{ex:first}.
Then $\rgin(I)=(x^3, x^2y^2, xy^4, y^6)$. 
See \cite{AbbottBigatti2016} for details about the computation of $\gin$
in CoCoA.
\end{example}

\begin{remark}\label{rem:genrgin}
  Let $I$ be a homogeneous ideal. If $I$ has a minimal generator
  of degree~$d$ (then so does $g(I)$),  then also
 $\rgin(I)$ has a minimal generator of  degree~$d$. 
 The converse is not true in general:
  consider for example the ideal  $I=(z^5,xyz^3)$ in $\QQ[x,y,z]$, then
  $\rgin(I)=(x^5,x^4y,x^3y^3)$ has a minimal generator of degree 6,
  and $I$ doesn't.
In particular, this shows that the highest degree of a minimal
generating set of $\rgin(I)$ may be strictly greater than that of $I$.
\end{remark}

The following result represents the
sectional matrix analogue of 
Macaulay's Theorem for Hilbert functions 
($H_{P/I} = H_{P/\LT(I)}$): 
it reduces the study of the sectional matrix of
a homogeneous ideal to the combinatorial behaviour of a monomial
ideal.

\begin{lemma}\label{lemma:rgin}
Let $I$ be a homogeneous ideal in $P = K[x_1,\dots, x_n]$. 
Then 
$$\M_{P/I} (i,d) = \M_{P/\rgin(I)}(i,d) =
\dim_K(P_d/(\rgin(I)+(x_{i+1},\dots,x_n))_d).$$
\end{lemma}
\begin{proof}
See Lemma~5.5 in \cite{bigatti1997borel}.
\end{proof}

\begin{remark}\label{rem:stronglystable}
  Lemma~\ref{lemma:rgin} shows that when we have a strongly stable
  ideal $J$ in $P$ (and in particular $\rgin(I)$ is strongly stable) the
  sectional matrix of ${P/J}$ is particularly easy to compute because
  sectioning $J$ by~$n{-}i$ generic linear forms is the same as
  sectioning $J$ by the smallest~$n{-}i$ indeterminates, $x_{i{+}1},
  \dots, x_n$.
\end{remark}

Using this combinatorial view, Bigatti and Robbiano in~\cite{bigatti1997borel} 
proved a combination of Macaulay's and Green's inequalities
(\cite{Macaulay27}, \cite{Green88}) and then
an analogue of Gotzmann's Persistence Theorem~(\cite{Gotzmann78})
for sectional matrices.

We recall the definition of binomial expansion
following the notation of \cite{bigatti1997borel}.
If $I$ is a homogeneous ideal then
the $(n-1)$-binomial expansion of $H_{I}(d)$ corresponds to a
``description'' of a lex-segment ideal $\mathcal{L}$ in degree~$d$,
and similarly the $d$-binomial expansion of $H_{P/I}(d)$ corresponds to 
$P/\mathcal{L}$.
See for example~\cite[Proposition~5.5.13]{KR2}.


\begin{definition}
For $h$ and $i$, two positive integers, we define the
\textbf{\boldmath $i$\nobreakdash-binomial expansion} of $h$ as
$
\textstyle
h= \binom{h(i)}{ i} + \binom {h(i-1)}{ i-1} + \dots + \binom {h(j)}{j}
$
with $h(i) >h(i-1) > \dots > h(j) \ge j \ge 1$.
Such expression exists and is unique. 

Moreover we define a family of functions
related to the expansion in the following way:
$\textstyle
( h_i ) ^s_t := \binom{h(i)+s}{ i+t} + \binom {h(i-1)+s}{ i-1+t} +
\dots + \binom {h(j)+s}{ j+t}
$.

For short, we will write 
$( h_i ) ^+_+$ instead of $( h_i ) ^{1}_{1}$,
and
$( h_i ) ^-$ instead of~$( h_i ) ^{-1}_0$.
\end{definition}

Here is Theorem~5.6  of \cite{bigatti1997borel} 
and again we convert the statement in terms of
the quotient $P/I$ 
using the properties of the 
functions derived from the binomial expansion.

\begin{theorem}[Sectional matrices, 1997]
\label{thm:HI_SM}
Let $I$ be a homogeneous ideal in the polynomial ring
$P=K[x_1,\dots,x_n]$ and $\M :=\M_{P/I}$. Then
\setlength{\leftmargini}{\parindent}
\begin{enumerate}
\item \label{item:HI_Sum}
      $\M(i,d+1) \le \sum\limits_{j=1}^i \M(j,d)$
      for all $i=1,\dots,n$ and $d\in\NN$.
\item \label{item:HI_Row}(Macaulay)
      $\M(i,d+1) \le (\M(i,d)_{d})^+_+$
      for all $i=1,\dots,n$ and $d\in\NN$.

\item \label{item:HI_Diff}
$\M(i-1,d) - \M(i-2,d) \le
   \left((\M(i,d)-\M(i-1,d))_{d}\right)^-$
      for all $i=3,\dots,n$ and $d\in\NN$.  
\item \label{item:HI_HSec}(Green)
      $\M(i-1,d) \le (\M(i,d)_{d})^-$
      for all $i=2,\dots,n$ and $d\in\NN$.
\end{enumerate}
\end{theorem}

\begin{proof}
 This is Theorem~5.6  of \cite{bigatti1997borel}, 
using the conversions 
$$(H_I(d)_{n-1})^+ = (H_{P/I}(d)_d)^+_+
\text{\quad and \quad}
(H_I(d)_{n-1})^-_- = (H_{P/I}(d)_d)^-
$$
(see for example \cite{KR2} Proposition~5.5.16 and Proposition~5.5.18).
\end{proof}



For the extremal case of Bigatti, Geramita and Migliore in \cite{bigatti1994geometric}, 
Macaulay's inequality  defined 
the \textit{maximal growth} of the Hilbert function.
We analogously define maximal growth of the sectional matrix
following the extremal case in Theorem~\ref{thm:HI_SM}.a.

\begin{definition}\label{def:maximal-growth}
Let $I$ be a homogeneous ideal in~$P=K[x_1,\dots,x_n]$.
\begin{itemize}
\item
The Hilbert function $H_{P/I}$ has  
\textbf{\boldmath maximal growth in degree~$d$}
  if ``Macaulay's equality'' holds:
$H_{P/I}(d+1) = (H_{P/I}(d)_d)^+_+.$
\item
The sectional matrix $\M_{P/I}$ has
\textbf{\boldmath $i$-maximal growth in degree~$d$}
if ``Bigatti-Robbiano's equality'' holds:
 $\M_{P/I}(i,d+1) = \sum\limits_{j=1}^i \M_{P/I}(j,d)$.
\end{itemize}
\end{definition}

\begin{remark}\label{rem:BReq_Nomingen} 
For a homogeneous ideal $I$  in $P=\KK[x_1, \ldots ,x_n]$ 
if $\M_{P/I}$ has $n$-maximal growth in degree $d$
then $I$, and $\rgin(I)$,
have no minimal generators of degree~$d{+}1$.

More precisely, for any $i\in\{1,\dots,n\}$,
Corollary~2.7 of \cite{bigatti1997borel} implies that
  $\M_{P/I}$ has $i$-maximal growth in degree $d$
 if and only if
  $\rgin(I)$ has no minimal generators of degree $d{+}1$ in $x_1,\dots, x_i$.
(This is a generalization of Lemma~2.17 in ~\cite{ahn2007some}.)
\end{remark}

\section{Sectional Persistence Theorem}
\label{persistence}

Gotzmann's Persistence Theorem \cite{Gotzmann78} 
says that, if the
generators of an ideal $I$ have degree $\le \delta$ and 
the Hilbert function of $P/I$ has maximal growth in degree $\delta$,
then it has maximal growth for all higher degrees.

This is also true for sectional matrices.
Here we recall the 
Persistence Theorem~5.8 of \cite{bigatti1997borel},
and in Theorem~\ref{thm:Sectional_Persistence} we will 
generalize it for $i$-maximal growth, for $i\le n$.

\begin{theorem}[Persistence Theorem, 1997]\label{thm:SM_Persistence}\ \\
Let $I$ be a homogeneous ideal in $P=K[x_1,\dots,x_n]$.
If $\M_{P/I}$ has $n$\nobreakdash-maximal growth in degree $\delta$
and  $I$ has no generators of degree ${>}\delta$
then it has $i$-maximal growth for all $i=1,\dots,n$ and 
for all degrees ${>}\delta$.

Moreover, $\M_{P/I}$ has $n$-maximal growth for all degrees $\ge\reg(I)$.
\end{theorem}

\begin{example}\label{ex:regexampletrunc} Consider the polynomial ring $P=\mathbb{Q}[x,y,z,t]$ and the ideal
$I=(x^4 -x^2yz, \; x^5 +xy^3z)$ of $P$. Then 
$$
\M_{P/I}=
\begin{array}{rcccccccccc} 
  _0 & _1 & _2 & _3  & _4 & _5 & _6 & _7 & _8 & \dots\\
  1 & 1 & 1 & 1 & 0 & 0 & 0 & 0 & 0 & \dots\\
  1 & 2 & 3 & 4 & 4 & 3 & 2 & 1 & 1& \dots\\
  1 & 3 & 6 & 10 & 14 & 17 & 19 & 20 & 21& \dots\\
  1  & 4 & 10 & 20 & 34 & 51 & 70 & 90 & \fbox{$111$}&\dots
\end{array} 
$$
Therefore $\M_{P/I}$ has $4$-maximal growth starting from degree~$7$, whereas a direct computation shows that
$H_{P/I}$ has maximal growth starting from degree~$49$.
%
\end{example}

\begin{remark}\label{rem:reg-trunc}
  In particular, the regularity is used by CoCoA for truncating the
  size of the sectional matrix, displaying the rows up to degree
  $\reg(I){+}1$, so that the last column shows the persisting equalities.
  In Example~\ref{ex:first}, we have $\rgin(I)=(x^3,x^2y^2,xy^4,y^6)$, thus $\reg(I)=6$, and in degree~$7$ we
  read the persisting equalities.
\end{remark}

\begin{remark}\label{rem:reg}
  We emphasize that the regularity of a homogeneous ideal $I$,
  the highest degree of the generators of $\rgin(I)$,
  is usually a much lower number
  than the highest degree of the generators of the lex-segment ideal with
  the same Hilbert function of $I$, as shown in Example~\ref{ex:regexampletrunc}.  This fact makes the persistence
  in Theorem~\ref{thm:SM_Persistence} more ``practical'' than Gotzmann's.
\end{remark}

With the next lemma we show that if $\M_{P/I}$ has $i$-maximal 
growth in degree~$\delta$ for some $i<n$, this persists in 
higher degrees, even if it does not have $n$-maximal growth.

\begin{lemma}\label{lemma:no_gens_i}
Let $I$ be a homogeneous ideal in $P=K[x_1,\dots,x_n]$ 
generated  in
 degree $\le \delta+1$. 
If there exists $i\le n$  such that
$\rgin(I)$ has no minimal generators of degree $\delta+1$ in
$P_{(i)}=K[x_1,\dots,x_i]$,
then $\rgin(I)$ has no minimal generators of any degree $>\delta$ in $P_{(i)}$.
\end{lemma}
\begin{proof} 
  Let $\sigma$ be DegRevLex and $g$ a generic change of coordinates.
  Suppose that the $\sigma$-Gr\"obner basis of $g(I)$ has a
  polynomial~$f_2$ of degree $\delta+2$.  Then~$f_2$ comes from a
  minimal syzygy of $\rgin(I)=\LT_\sigma(g(I))$ and hence
this syzygy is linear (see~Lemma~5.7 in~\cite{bigatti1997borel}).
  This means that there exists a minimal generator~$t_1$ of $\rgin(I)$ of degree
  $\delta+1$, and, by the hypothesis, all minimal generators of
  $\rgin(I)$ must be in the ideal~$(x_{i+1},\dots, x_n)$. Let $f_1$ be the
  Gr\"obner basis polynomial such that $t_1=\LT_\sigma(f_1)$.
  Because we are using the reverse lexicographic term-ordering, this
  fact implies that $f_1\in(x_{i+1},\dots, x_n)$.  As a consequence
  any s-polynomial constructed with $f_1$ is a difference of polynomials in
  $(x_{i+1},\dots, x_n)$, so $f_2\in(x_{i+1},\dots, x_n)$. 
 Thus any Gr\"obner basis element of degree $\delta+2$ is in $(x_{i+1},\dots,
  x_n)$, and therefore $\rgin(I)$ has no minimal generators of degree
  $\delta+2$ in $\KK[x_1,\dots,x_i]$.  Iterating this reasoning, we
  can conclude that $\rgin(I)$ has no minimal generators of degree
  $>\delta$ in $\KK[x_1,\dots,x_i]$.
\end{proof}

Using Lemma~\ref{lemma:no_gens_i}, 
we can now extend Theorem~\ref{thm:SM_Persistence}.

\begin{theorem}[Sectional Persistence Theorem]
\label{thm:Sectional_Persistence}
\ \\
Let $I$ be a homogeneous ideal in $P=K[x_1,\dots,x_n]$ 
generated  in
 degree~$\le \delta+1$. 
  If there exists 
 $i\in\{1,\dots, n\}$ such that
$\M_{P/I}$ has $i$-maximal growth in degree~$\delta$
then it has $j$-maximal growth for all $j\in\{1,\dots,i\}$ and 
for all degrees $\ge\delta$.
\end{theorem}
\begin{proof}
  From $\M_{P/I}(i,\delta+1) =\sum_{j=1}^i \M_{P/I}(j,\delta)$ we have that
  $\rgin(I)$ has no generators of degree $\delta+1$ in $K[x_1,\dots,x_i]$
  (see Remark~\ref{rem:BReq_Nomingen}).  Then, applying
  Lemma~\ref{lemma:no_gens_i}, we know that $\rgin(I)$ has no
  generators of degree $>\delta$ in $K[x_1,\dots,x_i]$.  Hence we can apply
  Theorem~\ref{thm:SM_Persistence} to $J_{(i)} = \rgin(I)\cap
  K[x_1,\dots,x_i]$ and for all $d>\delta$ and $j=1,\dots,i$ we get
  $\M_{P/\rgin(I)}(i, d{+}1) = \M_{P/J_{(i)}}(j, d{+}1) = \sum_{k=1}^j
  \M_{P/J_{(i)}}(k,d) = \sum_{k=1}^j \M_{P/\rgin(I)}(k,d)$.
The conclusion now follows from Lemma \ref{lemma:rgin}.
\end{proof}

The Sectional Persistence Theorem says that, whereas the persistence
of the $n$-th row starts at the regularity of the ideal, the
persistence in the first rows may be detected in degree lower than the
highest degree of the generators.

\begin{example}\label{ex:before-reg}
Consider the polynomial ring $P=\QQ[x,y,z,w]$ and the ideal
$I=(x^2, \;xy, \; xz(z+w), \;x(z^2+w^2))$ of $P$. Then $\rgin(I)=(x^2, xy, xz^2, xzw, xw^3)$.
Notice that $I$ is generated in degree 
$\le 3$, and its regularity is $4$, so  from
Theorem~\ref{thm:SM_Persistence}
it follows that
$\M_{P/I}$ has $i$-maximal growth in degree~$4$ for $i=1,\dots, 4$.
$$
\M_{P/I}=
\begin{array}{rcccccccccc} 
  _0 & _1 & _2 & _3  & _4 & _5 &  \dots\\
  1 & 1 & 0 & 0 & 0 & 0 &  \dots\\
  1 & 2 & 1 & 1 & 1 & 1 &  \dots\\
  1 & 3 & 4 & 4 & 5 & 6 &  \dots\\
  1  & 4 & 8 & 11 & 15 & 21 & \dots
\end{array} 
$$
We see that $\M_{P/I}$ has $2$-maximal growth in degree~$2$ and 
 $3$-maximal growth in degree~$3$:
\\
$\M_{P/I}(2,3) = 1 = \M_{P/I}(1,2) + \M_{P/I}(2,2)$\\
$\M_{P/I}(3,4) = 5 = \M_{P/I}(1,3) + \M_{P/I}(2,3) + \M_{P/I}(3,3)$.\\
Hence from Theorem~\ref{thm:Sectional_Persistence}
it follows that  $\M_{P/I}$ has $2$-maximal growth for all degrees~$\ge2$
and it has $3$-maximal growth for all degrees~$\ge3$.
\end{example}

\section{Hilbert Polynomial, Hilbert Series, dimension, multiplicity}
\label{HP,HS}

In this section, we show how to read some algebraic invariants for a
homogeneous ideal $I$ from its sectional matrix $\M_{P/I}$ truncated
at some degree $\delta$.


\begin{lemma}\label{lemma:SM_Persistence_explicit}
Let $I$ be a homogeneous ideal in $P=K[x_1, \dots,x_n]$ generated
in degree $\le\delta+1$.
If $\M_{P/I}$ has $i$-maximal growth in degree $\delta$, then
for all $d\in\NN_{>0}$ we have
$$\textstyle
  \M_{P/I}(i, \delta{+}d) =
  \sum_{j=1}^{i}
    \binom{i-j{+}d{-}1}{i-j}\cdot\M_{P/I}(j,\delta)
$$
\end{lemma}

\begin{proof}
We prove the statement by induction on $d$.
From Theorem~\ref{thm:Sectional_Persistence} it follows that 
$\M_{P/I}$ has $k$-maximal growth in degree $\delta$, for all
$k=1,\dots, i$, thus for $d=1$
$\M_{P/I}(k, \delta{+}1) =  \sum_{j=1}^{k} \M_{P/I}(j,\delta)$.

Let d>1 and suppose the stated equality holds for 
$\M_{P/I}(k, \delta+d)$ for all $k=1,\dots, i$. 
Then by the $i$-maximal growth from Theorem~\ref{thm:Sectional_Persistence},
$$\M_{P/I}(i, \delta+d+1) =
  \sum_{k=1}^{i}
    \M_{P/I}(k, \delta+d) =
  \sum_{k=1}^{i}
  \sum_{j=1}^{k}
    \binom{k-j{+}d{-}1}{k-j}\cdot\M_{P/I}(j,\delta)$$
then we swap the sums varying $j$ in $\{1,\dots,i\}$
and  $t = k-j$ with $k\in\{j,\dots,i\}$:

$
  \sum_{j=1}^{i}
\Big(  \sum_{t=0}^{i-j}
    \binom{t{+}d{-}1}{t}
\Big)
\cdot\M_{P/I}(j,\delta)
=
  \sum_{j=1}^{i}
    \binom{i-j{+}d{-}1}{i-j}\cdot\M_{P/I}(j,\delta)
$
\\
and this concludes the proof.
\end{proof}


\begin{proposition}\label{prop:SM_Persistence_HPoly}
Let $I$ be a homogeneous ideal in $P=K[x_1, \dots,x_n]$ generated
in degree $\le\delta+1$,
and let $L_1, \dots, L_n$ be generic linear forms in $P$.
If $\M_{P/I}$ has $i$-maximal growth in degree $\delta$, then
the Hilbert polynomial of
$P/(I + (L_1, \dots, L_{n-i}))$ is
\begin{eqnarray*}
   p_i(x) &=&
\sum_{j=1}^i
\binom{i{-}j{+}x{-}\delta{-}1}{i-j}\cdot\M_{P/I}(j,\delta)
 \end{eqnarray*}
In particular,
if $p_i\ne0$ 
let $k = \min\{j\in\{1,\dots,i\}\mid \M_{P/I}(j,\delta)\ne0\}$,
then 

$$p_i(x) = \frac{\M_{P/I}(k,\delta)}{(i-k)!}x^{i-k} + ... \text{ terms of lower
  degree }.$$
\end{proposition}

\begin{proof}
The first part of the corollary is trivial from
Lemma~\ref{lemma:SM_Persistence_explicit}:
for $x>\delta$ we have
$p_i(x) = \M_{P/I}(i,x) =
  \sum_{j=1}^{i}
    \binom{{i-j{+}(x-\delta){-}1}}{{i-j}}\cdot\M_{P/I}(j,\delta).
$

We conclude by observing that
$\binom{i{-}k{+}x{-}\delta{-}1}{i-k}$ =
$\frac{({x{-}\delta{+}i{-}k{-}1})
\dots
({x{-}\delta{+}1})
({x{-}\delta})
}{(i-k)!}$
and therefore equal to
$ \frac1{(i-k)!}x^{i-k}+ (\text{terms of lower degree})
$.
\end{proof}

\begin{proposition}\label{prop:HilbertSeries}
Let $I$ be a homogeneous ideal in $P=K[x_1, \dots,x_n]$ generated
in degree $\le\delta+1$,
and let $L_1, \dots, L_n$ be generic linear forms in $P$.
If $\M_{P/I}$ has $i$-maximal growth in degree $\delta$, then
the Hilbert series of
$R_i = P/(I + (L_1, .., L_{n-i}))$ is
$$\HS_{R_i}(t)
=
\sum_{d=0}^\delta\M_{P/I}(i,d)t^d
+\bigg(\sum_{j=1}^i\frac{\M_{P/I}(j,\delta)}{(1-t)^{i-j+1}}\bigg)t^{\delta+1}.$$
\end{proposition}

\begin{proof}
By definition
$\HS_{R_i}(t)
=
\sum_{d=0}^\infty\M_{P/I}(i,d)t^d
$.
By Lemma~\ref{lemma:SM_Persistence_explicit},
we have that
$$
\sum_{d=\delta+1}^\infty\M_{P/I}(i,d)t^d
=
\sum_{d=\delta+1}^\infty  
\Big(  \sum_{j=1}^{i}
    \binom{{i-j{+}d-\delta-1}}{{i-j}}\cdot\M_{P/I}(j,\delta)
\Big)
t^d
$$
swapping  the sums and letting $k=d-\delta-1$ it becomes
\\
$
=
  \sum_{j=1}^{i}
\Big(
\sum_{k=0}^\infty  
    \binom{{i-j{+}k}}{{i-j}}
t^k
\Big)\cdot\M_{P/I}(j,\delta)\cdot t^{\delta+1}
=
\bigg(\sum_{j=1}^i\frac{\M_{P/I}(j,\delta)}{(1-t)^{i-j+1}}\bigg)t^{\delta+1}$.
Therefore, we can conclude by adding the first 
part of the series, $\sum_{d=0}^\delta\M_{P/I}(i,d)t^d$.
\end{proof}

%

The following theorem shows that we can easily read the dimension
and the multiplicity of $P/I$ from its sectional matrix.
In particular, this information may be found in the $\delta$-th column 
with $\delta<\reg(I)$ (see Example~\ref{ex:before-reg-2}).

\begin{theorem}\label{thm:dim_deg} 
Let $I$ be a homogeneous ideal in $P=K[x_1, \dots,x_n]$ generated
in degree $\le\delta+1$ such that $I_\delta{\ne} P_\delta$ and let $i=\min\{j~|~\M_{P/I}(j,\delta){\ne}0\}$. If
$\M_{P/I}(i,\delta)=\M_{P/I}(i,\delta+1)$, then
$$\dim(P/I)=n{-}i{+}1 \text{\quad and \quad}\deg(P/I)=\M_{P/I}(i,\delta).$$
\end{theorem}

%
%
\begin{proof}
The hypothesis implies that $\M_{P/I}$ has $i$-maximal growth in
degree~$\delta$,
so, by Theorem~\ref{thm:Sectional_Persistence}, 
has $j$-maximal growth in
degree~$d$, for all $d>\delta$ and $j=1,\dots,i$;
this means that $i=\min_j\{\M_{P/I}(j,d)\ne0\}$ for all $d>\delta$
and $\M_{P/I}(i,d)=\M_{P/I}(i,\delta)$ for all $d\ge\delta$.
Now, let $\delta'\ge\delta$ such that
$\M_{P/I}$ has $n$-maximal growth in degree~$\delta'$, for example 
$\delta'=\max\{\delta, \reg(I)\}$.
Applying Proposition~\ref{prop:HilbertSeries},
and setting the highest power, ${(1-t)^{n{-}i{+}1}}$, as common
denominator, it follows that
$$\HS_{P/I}(t)=\frac{\M_{P/I}(i,\delta')+f(t)(1-t)}{(1-t)^{n{-}i{+}1}}$$
for some polynomial $f(t)\in K[t]$ and the fraction above is reduced.
Therefore, the degree of its denominator, $n{-}i{+}1$, is $\dim(P/I)$,
and the evaluation of the numerator in 1, $\M_{P/I}(i,\delta)$, is $\deg(P/I)$.
\end{proof}

\begin{remark} 
 The hypothesis of Theorem~\ref{thm:dim_deg} is
equivalent to the existence of 
an integer $i$ such that $\M_{P/I}(i,\delta)=\M_{P/I}(i,\delta+1)$ 
for $\delta> r_{n-i+1}(P/I)$:
this kind of formulation should look more familiar to 
the readers of \cite{bigatti1994geometric} and \cite{ahn2007some}.
\end{remark}

\begin{example}\label{ex:before-reg-2}
Following Example~\ref{ex:before-reg}
we consider $i=2,\, \delta=2$ and then
$\M_{P/I}(1,2)=0$ and 
$\M_{P/I}(2,2)=\M_{P/I}(2,3)=1\ne0$.
We then conclude 
$\dim(P/I)=n{-}i{+}1 = 4-2+1 = 3$ and
$\deg(P/I)=\M_{P/I}(i,\delta) = \M_{P/I}(2,2) = 1$.

Note that we deduced this information from the sectional matrix
in degree~$2$, strictly smaller than $\reg(I)=4$ and also smaller than $3$, the maximal
degree of the generators of $I$.
\end{example}

\begin{remark}
  For any homogeneous ideal $I$ in $P=K[x_1,\dots,x_n]$, $\M_{P/I}$ has
  $i$-maximal growth for all degrees $\ge\reg(I)$ and for all
  $i\in\{1,\dots,n\}$.  Therefore all the results in this section
  hold replacing their hypotheses with ``$\delta\ge\reg(I)$''.
\end{remark}

\begin{example}
In Example~\ref{ex:first}, with $\reg(I)=6$, 
we have $i=3$ so $\dim(P/I) =3-3+1=1$ and
$\deg(P/I)=\M_{P/I}(3,6)=12$, 
\end{example}

\section{Maximal Growth and Greatest Common Divisor}
\label{GCD}

Now we make a subtle change: we consider the same scenario in degree
$\delta$ for an arbitrary homogeneous ideal $I$ and see that the same
conclusion hold on the \textit{truncation} $\ideal{I_{\le \delta}}$.

\begin{proposition}\label{prop:dim-deg-trunc}  
Let $I$ be a homogeneous ideal in $P=K[x_1,\dots,x_n]$.
If there exists $\delta$ such that $I_\delta{\ne}P_\delta$
and $\M_{P/I}(i,\delta)=\M_{P/I}(i,\delta{+}1)$, for $i=\min\{j{>}1~|~\M_{P/I}(j,\delta)\ne0\}$, then we have 
$$\dim(P/\ideal{ I_{\leq \delta}})=\dim(P/\ideal{ I_{\leq \delta{+}1}})=n{-}i{+}1,$$  
$$\deg(P/\ideal{ I_{\leq \delta}})=\deg(P/\ideal{ I_{\leq \delta{+}1}})=\M_{P/I}(i,\delta).$$ 
    \end{proposition}
\begin{proof} 
By construction
 $\M_{P/I}(j,\delta)=\M_{P/\ideal{ I_{\leq \delta}}}(j,\delta)$ 
and $\M_{P/I}(j,\delta{+}1)=\M_{P/\ideal{ I_{\leq \delta{+}1}}}(j,\delta{+}1)$ for
all $j=1,\dots, n$. 
From Remark~\ref{rem:BReq_Nomingen} it follows that
$\rgin(I)$ has no minimal
generators in degree $\delta{+}1$ in $x_1,\dots, x_i$, and therefore nor does  
$\rgin(\ideal{ I_{\leq \delta}})\subseteq\rgin(I)$.
It then follows that $$\M_{P/I}(j,\delta{+}1)=\M_{P/\ideal{ I_{\leq \delta}}}(j,\delta{+}1)$$ for all
$j=1,\dots, i$. This implies that
$$\M_{P/\ideal{ I_{\leq \delta}}}(i,\delta{+}1) =
\M_{P/\ideal{ I_{\leq \delta}}}(i,\delta)\ne0,$$
$$\M_{P/\ideal{ I_{\leq \delta{+}1}}}(i,\delta{+}1) =
\M_{P/\ideal{ I_{\leq \delta{+}1}}}(i,\delta)\ne0,$$
and $i=\min\{j~|~\M_{P/\ideal{ I_{\leq \delta}}}(j,\delta)\ne0\}=\min\{j~|~\M_{P/\ideal{ I_{\leq \delta{+}1}}}(j,\delta)\ne0\}$.
Now we can apply Theorem~\ref{thm:dim_deg} and get the conclusions.
\end{proof}

\begin{corollary}\label{cor:ksubvar_BRequal} 
Let $I$ be a homogeneous ideal in $P=\KK[x_1 ,\dots, x_n]$. If there exists
$\delta$ such that $I_\delta{\ne}P_\delta$ and
$\M_{P/I}$ has $n$-maximal growth in degree $\delta$, 
let $i=\min\{j{>}1 \mid \M_{P/I}(j,\delta){\ne}0\}$, then
$\ideal{ I_{\leq \delta}}=\ideal{ I_{\leq \delta{+}1}}$,
$\dim(P/\ideal{ I_{\leq \delta}})=n{-}i{+}1$, and
$\deg(P/\ideal{ I_{\leq \delta}})=\M_{P/I}(i,\delta).$
\end{corollary} 

\begin{proof}
By hypothesis $\M_{P/I}$ has $n$-maximal growth in degree $\delta$, hence by Remark~\ref{rem:BReq_Nomingen} it follows that
$I$ has no minimal
generators in degree $\delta{+}1$, and therefore $\ideal{ I_{\leq \delta}}=\ideal{ I_{\leq \delta{+}1}}$.
Moreover, from Theorem~\ref{thm:SM_Persistence}
it has $i$-maximal growth in degree $\delta$.
Hence we have the equality
$\M_{P/I}(i,\delta{+}1) =\sum_{j=1}^i  \M_{P/I}(j,\delta) = \M_{P/I}(i,\delta) \ne 0$
and the conclusion follows from 
Proposition~\ref{prop:dim-deg-trunc}.
\end{proof}

\begin{example}\label{ex:dim-deg}
Consider the polynomial ring $P=\QQ[x,y,z,t]$ and the ideal
$I=(x^3, x^2y, xy^2, xyz^2,  xyzt^3)$ of $P$. 
Then 
$$
\M_{P/I}=
\begin{array}{rcccccccccc} 
  _0 & _1 & _2 & _3  & _4 & _5 & _6 & _7 & \dots\\
  1 & 1 & 1 & 0 & 0 & 0 & 0 & 0 &  \dots\\
  1 & 2 & 3 & 1 & 1 & 1 & 1 & 1 & \dots\\
  1 & 3 & 6 & 7 & 7 & 8 & 9 & 10 &  \dots\\
  1  & 4 & 10 & 17 & 24 & 32 & 40 & 50 &\dots
\end{array} 
$$
$$
\M_{P/I_{\leq 3}}=
\begin{array}{rcccccccccc} 
  _0 & _1 & _2 & _3  & _4 &  \dots\\
  1 & 1 & 1 & 0 & 0 &   \dots\\
  1 & 2 & 3 & 1 & 1 &  \dots\\
  1 & 3 & 6 & 7 & 8 &   \dots\\
  1  & 4 & 10 & 17 & 25 &\dots
\end{array} 
~~\M_{P/I_{\leq 4}}=
\begin{array}{rcccccccccc} 
  _0 & _1 & _2 & _3  & _4 &_5 &  \dots\\
  1 & 1 & 1 & 0 & 0 & 0 &   \dots\\
  1 & 2 & 3 & 1 & 1 & 1 & \dots\\
  1 & 3 & 6 & 7 & 7 & 8 &   \dots\\
  1  & 4 & 10 & 17 & 24 & 32 &\dots
\end{array} 
$$
%
%
 We see that $\M_{P/I}$ has $i=2$-maximal growth in degree~$\delta=3$.
Then by Proposition~\ref{prop:dim-deg-trunc}, we have that
$\dim(P/\ideal{I_{\le3}})=\dim(P/\ideal{I_{\le4}})= 3$, and  $\deg(P/\ideal{I_{\le3}})=\deg(P/\ideal{I_{\le4}})= 1$,
regardless what happens in $\M_{P/I}(j,\delta)$ for $j>i=2$.
\end{example}

The following proposition is the generalization of 
Proposition~1.6 of \cite{bigatti1994geometric}.

\begin{proposition}\label{prop:BRequal_d-reg} 
Let $I$ be a homogeneous ideal in $P=\KK[x_1, \ldots , x_n ]$.
If $\M_{P/I}$ has $n$-maximal growth in degree $\delta$, 
then $\reg(\ideal{ I_{\leq \delta} })\le \delta$.
\end{proposition}  
\begin{proof}
Let $\overline{I} = \langle I_{\leq \delta} \rangle$. 
  By construction $\M_{P/I}(j,d)=\M_{P/\overline{I}}(j,d)$ for all
  $j=1,\dots, n$ and $0\le d\le \delta$. By
  Remark~\ref{rem:BReq_Nomingen}, $I$ has no minimal generators in
  degree 
  $\delta{+}1$, and hence $\M_{P/I}(j,\delta{+}1)=\M_{P/\overline{I}}(j,\delta{+}1)$ for all
  $j=1,\dots, n$. This implies that 
 $\M_{P/\overline I}$ has $n$-maximal growth in degree $\delta$, and then, 
 by the Persistence Theorem~\ref{thm:SM_Persistence},
it has $n$-maximal growth in all degrees $>\delta$.
By Lemma~\ref{lemma:rgin},
  $\M_{P/\overline{I}}=\M_{P/\rgin(\overline{I})}$ and hence by
  Lemma~\ref{lemma:no_gens_i}, $\rgin(\overline{I})$ has no minimal
  generators  of degree $> \delta$, and then
  $\reg(\overline{I}) 
\le \delta$.
\end{proof}

In the rest of this section, we generalize some results of
  \cite{bigatti1994geometric} and \cite{ahn2007some} about the
  existence a common factor when there is a certain kind of maximal
  growth.  

\begin{corollary}\label{cor:GCD_2row} 
Let $I$ be a homogeneous ideal in $P=\KK[x_1 ,\dots, x_n]$.
If there exists $\delta$ such that $I_\delta\ne \{0\}$
  and $\M_{P/I}(2,\delta)=\M_{P/I}(2,\delta{+}1)$
  (i.e.~has 2-maximal growth in degree~$\delta$) then $\ideal{I_{\le\delta}}$ has a GCD
  of degree $\M_{P/I}(2,\delta)$. Furthermore, $\ideal{I_{\le\delta{+}1}}$ shares the same GCD.
\end{corollary}

\begin{proof}
From $I_\delta{\ne}\{0\}$ it follows that $x_1^\delta\in\rgin(I)$,
and then $\M_{P/I}(1,\delta)=0$. If $\M_{P/I}(2,\delta)=0$ then
$I_\delta=P_\delta$ has GCD $=1$ of degree~$0$. 
Otherwise, by  Proposition~\ref{prop:dim-deg-trunc} $\dim(P/\ideal{ I_{\le_\delta}})=\dim(P/\ideal{ I_{\le_\delta{+}1}})=n{-}1$ and
  $\deg(P/\ideal{ I_{\le_\delta}})=\deg(P/\ideal{ I_{\le_{\delta{+}1}}})=k=\M_{P/I}(2,\delta)$.
This means that $\ideal{ I_{\le_\delta}}$ defines a hypersurface of degree $k$, 
\ie~$\ideal{ I_{\le_\delta}}= (F)\cap J$ with
 $\dim(J)<n{-}1$ and  $\deg(F)=k$.
Therefore $\ideal{I_{\le\delta}}\subseteq (F)$
 as claimed. Similarly for $\ideal{ I_{\le_{\delta{+}1}}}$.
\end{proof}





Following the statement of Corollary~\ref{cor:GCD_2row},
 and along the line of ideas in \cite{bigatti1994geometric},
we give a new definition for the potential GCD,
based on the sectional matrix instead of the Hilbert function.

\begin{definition} 
  Let $I$ be a homogeneous ideal in $P=K[x_1 ,\dots, x_n]$
 such that  $I_\delta\ne\{0\}$.
The \textbf{$\M$-potential degree of the  GCD} of $I_\delta$ is
  $k=\M_{P/I}(2,\delta)$.
\end{definition}

The following corollary is the generalization of 
Proposition~2.7 of \cite{bigatti1994geometric} and of Corollary 5.2 of 
\cite{ahn2007some}.

\begin{corollary}\label{cor:GCD_BRequal} 
  Let $I$ be a homogeneous ideal in $P=K[x_1 ,\dots, x_n]$ and let
  $\delta$ be such that $I_\delta\neq \{0\}$.  
  Let $k$ be the $\M$-potential
  degree of the GCD of $I_\delta$.  If
  $\M_{P/I}$ has $i$-maximal growth in degree $\delta$ for some $i\ge2$,
  then $\ideal{I_{\le\delta}}$ and $\ideal{I_{\le{\delta{+}1}}}$  share the same GCD, $F$, of degree $k$. 
\end{corollary} 

\begin{proof}
By Theorem \ref{thm:Sectional_Persistence}, if $\M_{P/I}$
has $i$-maximal growth in degree~$\delta$, then it has $2$-maximal growth in degree~$\delta$.
Therefore we conclude by Corollary~\ref{cor:GCD_2row}.
\end{proof}

Let us see this result in action on an example.
\begin{example}
 Consider the polynomial ring $P=\QQ[x,y,z]$ and the ideal
 $I=(x^3+y^3, 
\; x^2+3xy+2y^2-xz-yz, 
\; x^4 +x^3y,
\; xy^4-16xyz^3,  
\break 
y^5-3xy^3z-4y^4z+12xyz^3-25y^3z^2+100yz^4)$ of $P$.
 Then 
 $$
 \M_{P/I}=
\begin{array}{rcccccccccc} 
  _0 & _1 & _2 & _3  & _4 & _5 & _6 & _7 & _8 & _9 & \dots\\
  1 & 1 & 0 & 0 & 0 & 0 & 0 & 0 & 0 & 0 & \dots\\
  1 & 2 & 2 & 1 & 1 & 0 & 0 & 0 & 0 & 0 & \dots\\
  1 & 3 & 5 & 6 & 6 & 4 & 3 & 2 & 1& 1 &\dots
\end{array} 
$$
We see that $\M_{P/I}$ has $2$-maximal growth in degree~$3$
and indeed both
 $I_3$ and $I_4$ have a GCD of degree~$k = \M_{P/I}(2,3)= \M_{P/I}(2,4)=1$. 
 Indeed, a direct computation shows that the GCD is $x+y$.
%
\end{example}

\section{Saturated ideals}
\label{saturated}

  


A homogeneous ideal $I$ in $P{=}K[x_1,\ldots,x_n]$ 
is \textbf{saturated} 
if the irrelevant
maximal ideal ${\mathfrak{m}}{=} ( x_1 , \ldots,x_n)$
 is not an associated prime ideal,
 i.e. $(I:\mathfrak{m})=I$. 
For any homogeneous ideal $I$ of $P$, the \textbf{saturation} of $I$,
denoted $I^{\sat}$, is defined by
$I^{\sat} := \{ f \in P\mid f\mathfrak{m}^{\ell} \subseteq
I \text{ for some integer } \ell \}$.

In this section we apply the results obtained previously to the case of saturated ideals.

\begin{remark}\label{rem:sat}
It is well known that for any homogeneous ideal $J$ 
there exists $\ell \in \NN$ such that $J^{\sat} = J:\mathfrak{m}^\ell$
and therefore   $J_d= (J^{\sat})_d$ for all $ d\gg 0$.
\end{remark}

\begin{remark}\label{rem:satginnox_n} (Bayer-Stillman) 
Let $I$ be a homogeneous ideal of $P$. Then $\rgin(I^{\sat})=\rgin(I)_{x_n\to0}$.
  This shows that if $I$ is saturated, then $\rgin(I)$ has no minimal
  generators involving $x_n$.
\end{remark}


%
%
%
%
%
%

\begin{lemma}\label{lemma:BR_n-1row}
Let $I$ be a saturated ideal in $P=K[x_1, \ldots , x_n]$.
Then $\M_{P/I}$ has $n$-maximal growth in degree~$\delta$
if and only if it has $(n{-}1)$-maximal growth in degree~$\delta$.
\end{lemma}
\begin{proof} By Theorem~\ref{thm:HI_SM}.(a)
 $n$-maximal growth implies $(n{-}1)$-maximal growth.

Suppose now 
 $\M_{P/I}$ has $(n{-}1)$-maximal growth.
By Remark~\ref{rem:stronglystable} this implies that $\rgin(I)$ has no
minimal generators in $x_1,\dots,x_{n-1}$ in degree $\delta{+}1$.
Moreover, since $\rgin(I)$ is saturated, 
from Remark~\ref{rem:satginnox_n}
there are no minimal generators divisible
by $x_n$.
With no minimal generators in degree $\delta{+}1$
$\M_{P/\rgin(I)}$, and therefore $\M_{P/I}$,
has also $n$-maximal growth.
\end{proof}

In general the truncation of a saturated ideal is not saturated, as the following example shows. To guarantee
that also the truncation is saturated we need some additional
hypothesis, 
see Lemma~\ref{lemma:saturato}.
\begin{example}\label{ex:robbiano} 
Consider the polynomial ring $P=\QQ[x,y,z,t]$ and the ideal 
$I=(yz -xt,  \;z^3 -yt^2,  \;xz^2 -y^2t,  \;y^3 -x^2z,\; x^3, \;x^2y^2)$. 
This ideal is saturated, however, a direct computation shows that the truncation $I_{\leq 3}$
is not saturated.
\end{example}

The following lemma is the generalization of 
Lemma~1.4 of \cite{bigatti1994geometric}.

\begin{lemma}\label{lemma:saturato} 
Let $I$ be a saturated ideal in $P=K[x_1, \ldots , x_n]$.
If $\M_{P/I}$ has $(n{-}1)$-maximal growth in degree $\delta$
then 
the ideal $\langle I_{\leq \delta} \rangle$ is saturated.
\end{lemma}

\begin{proof}
From Lemma~\ref{lemma:BR_n-1row} it follows that $\M_{P/I}$ has
  $n$-maximal growth in degree~$\delta$.
  
Let $\overline{I} = \langle I_{\leq \delta} \rangle$
and $\widetilde{I} = \overline{I}^{\sat}$.
Notice that we have that $\overline{I}_d \subseteq
\widetilde{I}_d$ for all $d\in\NN$.
We want to prove that our hypotheses imply $\overline{I}_d =
\widetilde{I}_d$ for all $d\in\NN$.

By Remark~\ref{rem:sat} we have that
$\overline{I}_d = \widetilde{I}_d$ for all $d\gg 0$.

Let $f\in\widetilde{I}_d$ be an element with
$d\le \delta$.  Then $f\mathfrak{m}^\ell\subseteq \overline{I}$ for
some integer~$\ell$.  Since $\overline{I}\subseteq I$, we have
$f\mathfrak{m}^\ell\subseteq I$, and, by hypothesis, $I$ is saturated,
therefore $f\in I$.  Now, $I$ and $\overline{I}$ coincide in degree $\le
\delta$, hence $f\in \overline{I}$. This shows $\overline{I}_d =
\widetilde{I}_d$ for all $d\le \delta$.

By Lemma~\ref{lemma:no_gens_i}, $I$ has no minimal generators in
degree $\delta{+}1$, and hence $I_d = \overline{I}_{d} = \widetilde{I}_{d}$ also
for $d=\delta{+}1$.

By contradiction, let $d> \delta{+}1$ be the biggest integer such that  
$\overline{I}_d \subsetneq \widetilde{I}_d$.
This means that in degree $d{+}1$
$$\M_{P/\overline{I}}(n,d{+}1)=\M_{P/\widetilde{I}}(n,d{+}1)$$
and in degree $d$
$$\M_{P/\overline{I}}(n,d)>\M_{P/\widetilde{I}}(n,d)$$
$$\text{and~~} \M_{P/\overline{I}}(j,d)\ge\M_{P/\widetilde{I}}(j,d),
\text{ for }j=1,\dots,n{-}1.$$
By definition $\overline{I}$ is generated in degree $\delta<d$, hence, by
Theorem~\ref{thm:SM_Persistence}, we have that
$$\M_{P/\overline{I}}(n,d{+}1)
=\sum_{j=1}^n\M_{P/\overline{I}}(j,d).$$ 
Now, using the equalities and inequalities above, we get
$$\M_{P/\widetilde{I}}(n,d{+}1) = \M_{P/\overline{I}}(n,d{+}1)
=\sum_{j=1}^n\M_{P/\overline{I}}(j,d) >\sum_{j=1}^n\M_{P/\widetilde{I}}(j,d).$$
This is impossible by the inequalities in Theorem~\ref{thm:HI_SM}.(a).
\end{proof}

%
%

%

Extending Corollary~\ref{cor:ksubvar_BRequal} to the case of saturated ideals, 
we can generalize Theorem 3.6 of \cite{ahn2007some}
and Theorem~3.6 of \cite{bigatti1994geometric}.

\begin{corollary}\label{cor:sat-dim}
Let $I$ be a saturated ideal in $P=\KK[x_1 ,\dots, x_n]$. If there exists
$\delta$ such that $I_\delta{\ne}P_\delta$ and
$\M_{P/I}$ has $(n{-}1)$-maximal or $n$\nobreakdash-maximal growth in degree $\delta$, let
$i=\min\{j{>}1 \mid \M_{P/I}(j,\delta){\ne}0\}$, 
then $\ideal{ I_{\leq \delta} }$ is a saturated ideal of
dimension $n{-}i$, of degree  $\M_{P/I}(i,\delta)$ and it is
$\delta$-regular. 
Moreover, $\dim(P/I)\le n{-}i$.
\end{corollary}  
\begin{proof} By Lemma \ref{lemma:BR_n-1row}, since $I$ is saturated, having $(n{-}1)$-maximal or $n$-maximal growth is equivalent. By Lemma~\ref{lemma:saturato}, $\langle I_{\leq \delta} \rangle$ is a saturated ideal. The conclusions then follows from Corollary~\ref{cor:ksubvar_BRequal} and Proposition~\ref{prop:BRequal_d-reg}.
\end{proof}

Now we can generalize
Corollary~5.2 of \cite{ahn2007some}
and Corollary~2.9 of \cite{bigatti1994geometric}.

\begin{corollary}\label{cor:n-1maxgrowthgcdsatur}
Let $I$ be a saturated ideal in $P=\KK[x_1 ,\dots, x_n]$.
If  $\M_{P/I}$ has
$(n{-}1)$-maximal growth in degree~$\delta$
and potential degree of the GCD = $k \geq 1$.  
Then  $\ideal{I_{\le\delta}}=\ideal{I_{\le{\delta{+}1}}}$ is saturated and it has a GCD of degree $k$.
\end{corollary}  
\begin{proof}
  From Lemma~\ref{lemma:BR_n-1row} it follows that $\M_{P/I}$ has
  $n$-maximal growth in degree~$\delta$.  Hence
  Corollary~\ref{cor:GCD_BRequal} and Lemma~\ref{lemma:saturato} apply.
\end{proof}

Similarly to Corollary~\ref{cor:sat-dim}, 
we might be tempted to extend
Proposition~\ref{prop:dim-deg-trunc} to the case of saturated ideals,
or, equivalently,
Corollary~\ref{cor:sat-dim} to the case of 2-maximal growth.
The example below shows that this is not possible.





\begin{example}\label{ex:truncatnosat} 
 As in \cite{ahn2007some}, under the assumption of 
Proposition~\ref{prop:dim-deg-trunc}, if $I$ is saturated
 we can not conclude that $\ideal{ I_{\leq \delta}}$ is saturated.
For this example, we consider a first set of $98$ points
on the conic $Q$ with equation $(z -3t)(z+3t)=0$ in $\mathbb{P}^3$ 
and a second set of $16$ points outside $Q$. In this way we obtain
a saturated homogeneous ideal $I$ in $P=\QQ[x,y,z,t]$ and 
$  \M_{P/I}$ is 
  $$
\begin{array}{rcccccccccccccccccccccc} 
  _0 & _1 & _2 & _3  & _4 & _5 & _6 & _7 & _8 & _9 & _{10} & _{11} & _{12} & _{13} & _{14} &  \dots\\
  1  & 1 &  1 &  1 &  1 &  0 &  0 &  0 &  0 &  0 &   0 &   0 &   0 &  0 &  0  &   \dots\\
  1 & 2 &  3 &  4 &  5 &  2 &  2 &  0 &  0 &  0 &   0 &   0 &   0 &   0 &   0 &    \dots\\
  1 & 3 &  6 & 10 & 15 & 17 & 13 & 13 & 11 &  9 &   7 &   5 &   3 &   1 &   0 &    \dots\\
  1 & 4 & 10 & 20 & 35 & 52 & 65 & 78 & 89 & 98 & 105 & 110 & 113 & 114 & 114 &  \dots\\
\end{array} 
$$
From $\M_{P/I}$ we can read that $\ideal{I_{\le5}}$ has a GCD
of degree~$2$ (Corollary~\ref{cor:GCD_2row}),
however $\M_{P/I}$ does not have $3$-maximal growth in degree~$5$ 
(so Corollary~\ref{cor:n-1maxgrowthgcdsatur} does not apply),
indeed a direct computation shows that $I_{\leq 5}$ is not saturated.
\end{example}

\begin{example} 
Consider the polynomial ring $P=\QQ[x,y,z,t,h]$ and the strongly stable ideal 
$$I=(x^5, x^4y, x^3y^2, x^2y^3, xy^4, x^4z, x^3yz, x^2y^2z, x^4t,
\;\;xy^3z^3).$$ 

The ideal $I$ is saturated and 
  $$
  \M_{P/I}=
\begin{array}{rcccccccccc} 
  _0 & _1 & _2 & _3  & _4 & _5 & _6 & _7 & _8 &  \dots\\
  1 & 1 &  1 &  1 &  1 &   0 &   0 &   0 &   0 & \dots\\
  1 & 2 &  3 &  4 &  5 &   1 &   1 &   1 &   1 & \dots\\
  1 & 3 &  6 & 10 & 15 &  13 &  14 &  14 &  15 & \dots\\
  1 & 4 & 10 & 20 & 35 &  47 &  61 &  75 &  90 & \dots\\
  1 & 5 & 15 & 35 & 70 & 117 & 178 & 253 & 343 & \dots
\end{array} 
$$
We can now apply Corollary~\ref{cor:n-1maxgrowthgcdsatur} and see that
$\ideal{I_{\le5}}$ and $\ideal{I_{\le6}}$ are saturated and have GCD
of degree $1$ (we can check that the GCD is $x$).  For this example the
result in \cite[Corollary 5.2]{ahn2007some}, for detecting a GCD, do
not apply.
\end{example}

\section{Sectional matrices, GIN, and resolutions}
\label{examples}

 In this section, we will present some examples in order to compare the sectional matrix with other algebraic invariants, 
 such as the Hilbert function $H$, the generic initial ideal and the minimal resolution.
 
 We start from two homogeneous ideals with same Hilbert function but different $\rgin$, sectional matrix and Betti numbers.
 
 \begin{example}\label{ex:8.1}
  Consider  $P=\QQ[x,y,z]$ and let $$I=(x^2, xy, xz, y^3, y^2z,yz^2, z^3) ~~~\text{ and }~~~ J=(x^2, xy, y^2, xz^2, yz^2, z^3)$$ be two ideals in $P$. 
  Both ideals are strongly stable and hence, they coincide with their own $\rgin$.
  These two ideals clearly have distinct $\rgin$, but they have the
  same Hilbert function (the last row in the sectional matrices). 
They have different sectional matrix  and different graded Betti numbers.  
  $$
  \M_{P/I}=
\begin{array}{rcccccccccc} 
  _0 & _1 & _2 & _3  & _4 &  \dots\\
  1 & 1 & 0 & 0 & 0 & \dots\\
  1 & 2 & 1 & 0 & 0 & \dots\\
  1 & 3 & 3 & 0 & 0 & \dots
\end{array} 
    ~~\M_{P/J}=
\begin{array}{rcccccccccc} 
  _0 & _1 & _2 & _3  & _4 &  \dots\\
  1 & 1 & 0 & 0 & 0 & \dots\\
  1 & 2 & 0 & 0 & 0 & \dots\\
  1 & 3 & 3 & 0 & 0 & \dots
\end{array} 
$$
The resolutions of $P/I$ and $P/J$ are respectively
$$0 \to P(-4)\oplus P(-5)^3 \to P(-3)^3\oplus P(-4)^7 \to P(-2)^3\oplus P(-3)^4 \to P\to P/I\to 0.$$
$$0 \to P(-5)^3 \to P(-3)^2\oplus P(-4)^6\to P(-2)^3\oplus P(-3)^3 \to P\to P/J\to0.$$                                 

%
%
 \end{example}
 
 In the following example, we show two ideals with the same sectional matrix and same Betti numbers, but different generic initial ideal.
 
 \begin{example}\label{ex:8.2}
  Consider the polynomial ring $P=\QQ[x,y,z]$ and the ideals 
  $$I=(x^5, x^4y, x^3y^2,x^2y^3, xy^4, x^4z, x^3yz, x^2y^2z, x^3z^2,
  \underline{x^2yz^2})$$ 
  $$J=(x^5, x^4y, x^3y^2,x^2y^3, xy^4, x^4z, x^3yz, x^2y^2z, x^3z^2,
  \underline{ xy^3z}).$$ 
  Both the ideals are strongly stable and hence, they coincide with their own $\rgin$.
  These two ideals clearly have distinct $\rgin$, but they have the same Hilbert function, the same sectional matrix and the same Betti numbers.
    $$
    \M_{P/I}=\M_{P/J}=
\begin{array}{rcccccccccc} 
  _0 & _1 & _2 & _3  & _4 & _5 & _6 &  \dots\\
  1 & 1 & 1 &  1 &  1 &  0 &  0 & \dots\\
  1 & 2 & 3 &  4 &  5 &  1 &  1 & \dots\\
  1 & 3 & 6 & 10 & 15 & 11 & 12 & \dots
\end{array} 
$$
The resolution of $P/I$ and $P/J$ is
$$0 \to P(-7)^5 \to P(-6)^{14} \to P(-5)^{10} \to P. 
 $$
%
 \end{example}

In the following example we show two ideals with the same
sectional matrix, but different generic initial ideal and different Betti numbers.
 \begin{example}\label{ex:8.3}
Consider the polynomial ring $P=\QQ[x,y,z]$ and the ideals 
  $$I=(x^5, x^4y, x^3y^2,x^2y^3, xy^4, x^4z, x^2y^2z, x^3z^2, \underline{x^2yz^2})$$
  $$J=(x^5, x^4y, x^3y^2,x^2y^3, xy^4, x^4z, x^2y^2z, x^3z^2, \underline{xy^3z}).$$ 
  We have that 
  $$\rgin(I)=(x^5, x^4y, x^3y^2, x^2y^3, xy^4, x^4z, 
      x^3yz, x^2y^2z, \underline{x^3z^2}, \underline{x^2yz^3}),$$
   $$\rgin(J)=(x^5, x^4y, x^3y^2, x^2y^3, xy^4, x^4z, x^3yz, 
      x^2y^2z, \underline{xy^3z}, \underline{x^3z^3}).$$
          $$
          \M_{P/I}=\M_{P/J}=
\begin{array}{rcccccccccc} 
  _0 & _1 & _2 & _3  & _4 & _5 & _6 & _7&  \dots\\
  1 & 1 & 1 &  1 &  1 &  0 &  0 &  0 & \dots\\
  1 & 2 & 3 &  4 &  5 &  1 &  1 &  1 & \dots\\
  1 & 3 & 6 & 10 & 15 & 12 & 12 & 13 & \dots
\end{array} 
$$
The resolution of $P/I$ is 
$$0 \to P(-7)^2\oplus P(-8) \to P(-6)^{11} \to P(-5)^9 \to P\to P/I\to 0.$$
The resolution of $P/J$ is
$$0 \to P(-7)^3\oplus P(-8) \to P(-6)^{11}\oplus P(-7) \to P(-5)^9 \to P\to P/J\to 0.$$
%
\end{example}
 
 In the following example we show two ideals with the same
 $\rgin$, therefore the same sectional matrix and  Hilbert function,
 but different Betti numbers. 
 
 \begin{example}\label{ex:8.4}
  Consider the polynomial ring $P=\QQ[x,y,z]$, and the ideals of $P$
  $$I=(x^4, y^4, z^4, xy^2z^3, x^3yz^2, x^2y^3z) ~~~\text{ and }~~~ J=\rgin(I).$$ 
  These two ideals clearly have the same $\rgin$, therefore the same sectional matrix.
  However, $J$ has more minimal generators than $I$, 
  so they have different resolutions.
          $$
          \M_{P/I}=\M_{P/J}=
\begin{array}{rcccccccccc} 
  _0 & _1 & _2 & _3  & _4 & _5 & _6 & _7&_8&  \dots\\
  1 & 1 & 1 &  1 &  0 &  0 &  0 &  0 & 0 & \dots\\
  1 & 2 & 3 &  4 &  2 &  0 & 0 & 0 & 0 & \dots\\
  1 & 3 & 6 & 10 & 12 & 12 & 7 & 0 & 0 & \dots
\end{array} 
$$
The resolution of $P/I$ is
$$0 \to P(-9)^7 \to P(-7)^3\oplus P(-8)^9 
                           \to P(-4)^3\oplus P(-6)^3 \to P\to P/I\to0. $$
The resolution of $P/J$ is
$$0 \to P(-8)^5\oplus P(-9)^7
  \to P(-5)^2\oplus P(-6)^2\oplus P(-7)^{10}\oplus P(-8)^{14}\to$$
  $$\to P(-4)^3\oplus P(-5)^2\oplus P(-6)^5\oplus P(-7)^7 \to P\to P/J\to 0. $$
  
%
 \end{example}

In summary: 
\begin{table}[htbp]
\centering
\begin{tabular}{|c|r|r|r|}
\hline
\multicolumn{1}{|c|}
{\textbf{Example}} &
{\textbf{$\rgin$}} &
{\textbf{Sec. Mat.}} &
{\textbf{Betti n.}} \\
\hline
\ref{ex:8.1} & $\ne$  & $\ne$ & $\ne$ \\
\hline
\ref{ex:8.2} & $\ne$ & $=$ & $=$  \\
\hline
\ref{ex:8.3} & $\ne$  & $=$ & $\ne$  \\
\hline
\ref{ex:8.4} & $=$ & $=$ & $\ne$  \\
\hline
\end{tabular}
\end{table}

%


\paragraph{\textbf{Acknowledgements}} The authors thank Professor
L. Robbiano for the valuable discussions and for suggesting us Example \ref{ex:robbiano}.

During the preparation of this paper the third author was supported by the MEXT grant for Tenure Tracking system.

\bibliography{BPT2016file}{}
\bibliographystyle{plain}

\end{document}